\documentclass[notitlepage,leqno,11pt]{article}
\usepackage{amssymb}
\catcode`\@=11 \@addtoreset{equation}{section}

\catcode`\@=12

\usepackage{latexsym}
\usepackage{textcomp}
\usepackage{amsmath}
\usepackage{amsfonts}
\usepackage{amssymb}
\usepackage{mathrsfs}
\usepackage{cite}
\newtheorem{Theorem}{Theorem}[section]

\newtheorem{corollary}[Theorem]{Corollary}

\newtheorem{lemma}[Theorem]{Lemma}

\newtheorem{Lemma}[Theorem]{Lemma}

\newtheorem{Corollary}[Theorem]{Corollary}

\newtheorem{Definition}[Theorem]{Definition}

\numberwithin{equation}{section}

\def\eps{\varepsilon}

\newcommand{\R}{\mathbb{R}}

\newcommand{\N}{\mathbb{N}}

\newcommand{\B}{{\bf B}}
\newcommand{\cA}{{\mathcal A}}

\newcommand{\cC}{{\mathcal C}}
\newcommand{\cD}{{\mathcal D}}

\newcommand{\cL}{{\mathcal L}}

\renewcommand{\phi}{\varphi}

\renewcommand{\a }{\alpha }

\newcommand{\e }{\varepsilon }

\newcommand{\G }{\Gamma}

\newcommand{\n }{\nabla }
\newcommand{\vp }{\varphi }
\renewcommand{\phi}{\varphi}

\renewcommand{\O }{\Omega }

\newcommand{\be}{\begin{equation}}
\newcommand{\ee}{\end{equation}}
\newcommand{\de}{\partial}

\newcommand{\Ds}{(-\Delta)^s}

\newenvironment{altproof}[1]
{\noindent
{\bf Proof of {#1}}.}
{\nopagebreak\mbox{}\hfill $\Box$\par\addvspace{0.5cm}}

\title{Monotonicity and nonexistence results for some fractional elliptic
problems in the half space}
\author{Mouhamed Moustapha Fall   and   Tobias Weth}

\begin{document}
\date{}
\maketitle
 \let\thefootnote\relax\footnotetext{mouhamed.m.fall@aims-senegal.org (M. M. Fall), weth@math.uni-frankfurt (T. Weth).}
\let\thefootnote\relax\footnotetext{Goethe-Universit\"{a}t Frankfurt, Institut f\"{u}r Mathematik.
Robert-Mayer-Str. 10 D-60054 Frankfurt, Germany.}
\let\thefootnote\relax\footnotetext{ African Institute
for Mathematical Sciences  of Senegal. Km 2, Route de Joal. BP 1418
Mbour, Senegal.}
\bigskip

\begin{abstract}
\noindent We study a class of fractional elliptic problems of the
form $\Ds u=  f(u)$ in the half space $\R^N_+:=\{x \in \R^N\::\:
x_1>0\}$ with the
complementary Dirichlet condition $u \equiv 0$ in $\R^N \setminus
\R^N_+$. Under mild assumptions on the nonlinearity $f$, we show that
bounded positive solutions are increasing in $x_1$. For the special
case $f(u)=u^q$, we deduce nonexistence of positive bounded solutions
in the case where $q \ge 1$ and $q<\frac{N-1+2s}{N-1-2s}$ if $N \ge 1+2s$.
We do not require integrability assumptions on the solutions we study.
\end{abstract}

\section{Introduction}
\label{sec:introduction}
In the present paper, we are concerned with solutions $u\in
L^\infty(\R^N)$ of the semilinear fractional problem
\begin{equation}
  \label{eq:main-half-space}
\left\{
  \begin{aligned}
  &\Ds u=  f(u),  \quad u \ge 0 &&\qquad\text{in $\R^N_+$},\\
  &u =0 && \qquad \text{in $\R^N \setminus \R^N_+$.}
 \end{aligned}
\right.
\end{equation}
Here $s\in(0,1)$,  $N\in\N$, $\R^N_+:= \{x \in \R^N\::\: x_1>0\}$
and $f: [0,\infty)  \to [0,\infty)$ is a nonnegative, nondecreasing
and locally Lipschitz continuous nonlinearity. Special attention
will be given to the case $f(t)=t^q$ with $q > 1$. Due to
applications in physics, biology and finance, linear and nonlinear
equations involving the fractional Laplacian $\Ds$ have received
growing attention in recent years (see e.g.
\cite[Introduction]{NPV11} for various references), while they are
are still much less understood than their non-fractional
counterparts.

We briefly explain
in which sense we consider problems of type (\ref{eq:main-half-space}). For functions
$u\in C^2_c(\R^N)$, the fractional Laplacian $(-\Delta)^s$ is defined by
\begin{equation}
  \label{eq:3}
\Ds u(x)=a_{N,s}\lim_{\e\to0}\int_{|x-y|>\e}\frac{u(x)-u(y)}{|x-y|^{N+2s}}\,dy,
\end{equation}
where $a_{N,s}=s(1-s)\pi^{-N/2}4^s\frac{\Gamma(\frac{N}{2}+s)}{\Gamma(2-s)}$
(see e.g. \cite[Remark 3.11]{cabre-sire}). Let $\cL^1_s$ denote the
space of all functions $u:\R^N\to
\R$ such that $\int_{\R^N}\frac{|u(x)|}{1+|x|^{N+2s}}dx<\infty$. If $\O
\subset \R^N$ is an open subset and $g \in L^1_{loc}(\O)$, we say that
a function
$v \in \cL^1_s$ solves the equation $\Ds v=g$ in $\O$ in the sense of
distributions if
$$
\int_{\R^N}v\Ds\phi\,dx=\int_{\O}g\phi\,dx \qquad \text{for every $\phi
  \in C^2_c(\Omega)$.}
$$
Note that the integral on the left hand side is well defined since
\be
\label{sec:introduction-4}
|\Ds
\phi(x)|\leq \frac{\kappa\|\phi\|_{C^2_c}}{1+|x|^{N+2s}} \qquad \text{for every $\phi
  \in C^2_c(\Omega)$, $x \in \R^N$}
\ee
with a constant $\kappa= \kappa(N,s,\Omega)$ (see for instance
\cite{Fall-Weth}). The following is our first  main result:
\begin{Theorem}\label{main-result-1}
Suppose that $f: [0,\infty)  \to
[0,\infty)$ is a nonnegative, nondecreasing
and locally Lipschitz continuous function satisfying $f(t)>0$ for
$t>0$ and
\begin{equation}
  \label{eq:extra-assumption}
\lim_{\stackrel{r,t \to 0}{r \not=t}}\frac{f(r)-f(t)}{r-t} = 0.
\end{equation}
Then every bounded solution $u$ of \eqref{eq:main-half-space} is
increasing in $x_1$. Moreover, either $u \equiv 0$, or $u$ is strictly
increasing in $x_1$.
\end{Theorem}

We note that, for $C^1$-nonlinearities, condition
(\ref{eq:extra-assumption}) simply amounts to $f'(0)=0$.  Our second
main result is of Liouville type.

\begin{Theorem}\label{main-result-2}
If $N \le 1+2s$ and $q>1$ or $N > 1+2s$ and $1<q<
\frac{N-1+2s}{N-1-2s}$, then the problem
\begin{equation}
  \label{eq:main-liouville}
\left\{
  \begin{aligned}
    &\Ds u=  u^q,  \quad u \ge 0 &&\qquad\text{in $\R^N_+$},\\
  &u =0 && \qquad \text{in $\R^N \setminus \R^N_+$.}
 \end{aligned}
\right.
\end{equation}
only admits the trivial solution $u \equiv 0$.
\end{Theorem}

Our results complement the following recent Liouville type result of
Jin, Li and Xiong \cite{JLX} for the corresponding full space problem
\begin{equation}
\label{eq:main-liouville-rn}
\Ds u=  u^q, \qquad u>0 \qquad \text{in $\R^N$.}
\end{equation}

\begin{Theorem}\label{main-result-3} (see \cite{JLX})\\
Suppose that $N\le 2s$ and $q>0$ or $N> 2s$ and $0<q< \frac{N+2s}{N-2s}$. Then
~(\ref{eq:main-liouville-rn}) has no bounded solution.
\end{Theorem}

We note that this result has also been obtained independently in
\cite{colorado-et-al} in the case $s \ge \frac{1}{2}$. Before that,
the special case $s=\frac{1}{2}$ had been considered in
\cite{Ou}, whereas in
\cite{Chen_Li_Ou05} the result was proved for a restricted class of solutions.\\

To put our results into perspective, some remarks are in order. Theorem~\ref{main-result-2} is an
improvement of \cite[Corollary 1.6]{Fall-Weth}, where the authors
established nonexistence of a restricted class of solutions $u$ of (\ref{eq:main-liouville})
in the subcritical case $N > 2s$ and $1<q\leq
\frac{N+2s}{N-2s}$. More precisely, in \cite{Fall-Weth} we assumed
that $u$ is contained in the Sobolev space $\cD^{s,2}(\R^N_+)$ defined
as the completion of $C_c^\infty(\R^N)$ with respect to the
norm given by
$$
\|u\|^2= \int_{\R^2N}\frac{(u(x)-u(y))^2}{|x-y|^{N+2s}}\,dxdy.
$$
The argument of \cite{Fall-Weth}, relying on the method of moving
spheres, does not apply under the assumptions of
Theorem~\ref{main-result-2}.\\
Theorem~\ref{main-result-1} is proved by a variant of the moving
plane method based on extensions and modifications of techniques in
the papers \cite{BCN-1996} resp. \cite{RW}, which were devoted to
second order and polyharmonic boundary value problems, respectively.
The first key step in the argument is to show, without a priori
integrability assumptions, that bounded solutions of
(\ref{eq:main-half-space}) admit a Green function representation.
This representation is obtained,  via an approximation argument,
from Green-Poisson type formulas in balls. Once the Green function
representation is obtained, we carry out a moving plane argument for
integral equations. We note that moving plane arguments for integral
equations have been applied very successfully in recent years, see
e.g. \cite{Chang_Yang, BGW,Chen_Li_Ou,Chen_Li_Ou05,BiLoWa,RW}. In
the present situation, the lack of integrability assumptions creates
additional difficulties which require to argue
somewhat differently than in earlier papers.\\
Theorem~\ref{main-result-2} is deduced from Theorems
\ref{main-result-1} and \ref{main-result-3} by considering the limits
of solutions of (\ref{eq:main-liouville}) as $x_1 \to \infty$, which,
considered as functions of $(x_2,\dots,x_N)$, solve
(\ref{eq:main-liouville-rn}) in $\R^{N-1}$.\\
The boundedness assumption in Theorem \ref{main-result-2} can be
replaced by only assuming boundedness in compact subsets of $\R^N_+$
if the assumption on $q$ is strengthened to $1<q< \frac{N+2s}{N-2s}$
in case $N > 2s$. This can be deduced from Theorems
\ref{main-result-2} and \ref{main-result-3} by the argument used in
\cite[Section 4]{RW1} for the polyharmonic version of
(\ref{eq:main-liouville}). The argument is based on the doubling-lemma
(see \cite{PQS}).\\
The combination of the Liouville type results
Theorem~\ref{main-result-2} and \ref{main-result-3} are expected to
give rise, via a Gidas-Spruck type rescaling argument (see
\cite{GS1}), to a priori bounds for solutions to more general
integral equations in bounded domains and also to elliptic boundary
value problems of second order with mixed nonlinear boundary
conditions. For applications of this type, it is essential that
Theorems \ref{main-result-2} and \ref{eq:main-liouville} do not
contain a priori integrability assumptions. This topic will be
considered by the authors in a future work.\\
The combination of Liouville
type results with rescaling arguments has already been applied
successfully by Cabr\'e and Tan \cite{cabre-tan} for nonlinear
boundary value problems involving the spectral theoretic square root
of the Dirichlet Laplacian, denoted by $\cA_{1/2}$ in
\cite{cabre-tan}. In particular, the analogue of
Theorem~\ref{main-result-2} with $\Ds$ replaced by $\cA_{1/2}$ has
been proved in \cite[Theorem 1.5]{cabre-tan}. The subtle
differences between $\Ds$ and spectral theoretic powers of the
Dirichlet Laplacian are discussed in
\cite[Remark 0.4]{fall} from a PDE point of view and in \cite{SV03} in terms of
stochastic processes. Because of these differences, it remains
unclear whether Theorem~\ref{main-result-2} can also be obtained via
 similar methods as in \cite{cabre-tan}. The approach of the present
 paper is completely different. Moreover, the monotonicity
 result given by Theorem~\ref{main-result-1} is not available yet for the
corresponding problem with spectral theoretic powers.\\

The paper is organized as follows. In Section~\ref{sec:preliminaries}
we collect preliminary results on (distributional) solutions of $\Ds u = f$ on some open subset of $\R^N$ with bounded
$f$. In Section~\ref{sec:representation} we show that bounded
solutions $u$ of the problem $\Ds u =f$ in $\R^N_+$ with $u \equiv 0$
in $\R^N \setminus \R^N_+$ admit a Green function representation
whenever $f$ is bounded and nonnegative. In Section~\ref{sec:proof-monot-result} we
complete the proof of Theorem~\ref{main-result-1}, and in
Section~\ref{sec:proof-theorem} we complete the proof of
Theorem~\ref{main-result-2}. The appendix contains a
regularity result needed in the proof of Theorem~\ref{main-result-1}.\\

\textbf{Acknowledgment:} The second author wishes to thank Enrico
Valdinoci und Eduardo Colorado for helpful discussions. The first
author is funded by the Alexander von Humboldt foundation and would
like to thank Krzysztof Bogdan for useful discussions.

\section{Preliminaries}
\label{sec:preliminaries}
Here and in the following, we consider $N\geq 1$ and $s\in(0,1)$. We
write $\B=\{x\in\R^N: |x|<1\}$ for the open unit ball in $\R^N$ and
set $B_R:= \{x \in \R^N: |x|<R\}$ for $R>0$. We start by recalling the
following estimate, see \cite{Fall-Weth}.
\begin{Lemma}\label{lem:bondDsvp}
Let $\O \subset \R^N$ be a bounded open set.  Then there
 exists a constant $C=C(N,s,\O)>0$ such that for all  $\vp\in C^2_c(\O)$ and for all  $x\in \R^N$
 \begin{equation}
   \label{eq:24}
\left|\int_{|x-y|>\e}\frac{\vp(x)-\vp(y)}{|x-y|^{N+2s}}dy  \right|\leq
\frac{C\|\vp\|_{C^2(\Omega)}}{1+|x|^{N+2s}} \qquad \text{for all $\e\in(0,1).$}
 \end{equation}
\end{Lemma}
As a consequence, $\Ds u$ can be defined for functions $u \in \cL^1_s$
in the following way in distributional sense. Here we recall that
$\cL^1_s$ is the space of all functions $u:\R^N\to
\R$ such that $\int_{\R^N}\frac{|u(x)|}{1+|x|^{N+2s}}dx<\infty$.

\begin{Definition}\label{def:fraclap-non-loc}
Let $\O$ be an open subset of $\R^N$. Given $u\in \cL^1_s$, the distribution  $\Ds u\in \cD'(\O)$
is defined as
$$
\langle\Ds u,\vp \rangle =\int_{\R^N} u\Ds\vp dx \qquad \text{for all $\vp \in C^\infty_c(\O).$}
$$
\end{Definition}

The following result is contained in \cite[Lemma 3.8]{BB-Schr}.

\begin{Lemma}\label{lem:discont}
Let $\O \subset \R^N$ be open. If $u\in \cL^1_s\cap C^2(\O)$, then the limit
$\lim \limits_{\e\to0}\int_{|x-y|>\e}\frac{u(x)-u(y)}{|x-y|^{N+2s}}dy$
exists for every $x \in \O$. Moreover, $\Ds u$ is a regular
distribution given by
\begin{equation}
  \label{eq:10}
\Ds u(x)=a_{N,s }
\lim_{\e\to0}\int_{|x-y|>\e}\frac{u(x)-u(y)}{|x-y|^{N+2s}}dy \qquad
\text{for $x \in \Omega$.}
\end{equation}
\end{Lemma}
In the situation of this lemma, $\Ds u$ is a priori only well defined
a.e. in $\O$ as a regular distribution. Nevertheless, under these
hypotheses, we may assume in the following that $\Ds u$ is defined
pointwise by (\ref{eq:10}) in all of $\O$.

\begin{Lemma}\label{lem:mmpC2}
Let $\O \subset \R^N$ be open, and suppose that $u\in  \cL^1_s\cap
C^2(\O)$ satisfies $\Ds u \le 0 $ in $\O$. Suppose furthermore that
there exists $x_0 \in \Omega$ such that $u(x_0) \ge u(y)$ for a.e. $y
\in \R^N$. Then $u(y)= u(x_0)$ for every $y \in \Omega$ and a.e. $y
\in \R^N \setminus \Omega$.
\end{Lemma}

\begin{proof}
Let $x_0\in\O$ satisfy $u(x_0)\geq u(y)$ for a.e. $y\in\R^N$. Then the
function
$$
\eps \to h(\eps):= \int_{|x-y|>\eps}
\frac{u(x_0)-u(y)}{|x_0-y|^{N+2s}}\,dy
$$
is nonnegative and nonincreasing in $(0,\infty)$, whereas
$\lim \limits_{\eps \to 0}h(\eps) \le 0$ by assumption and (\ref{eq:10}).
Hence $h \equiv 0$ in $(0,\infty)$, which, since $u \in C^2(\Omega)$,
shows that $u(y) = u(x_0)$ for every $y \in \Omega$ and a.e. $y \in
\R^N \setminus \Omega$.
\end{proof}

Next we consider the Poisson kernel of $B_R$ (see \cite{BGR}) which is given by
\begin{equation}\label{eq:Poisson-ker-BR}
\G_R(x,y)=C_{N,s}\left(\frac{R^2-|x|^2}{|y|^2-R^2} \right)^s|x-y|^{-N},\qquad |y|>R,\,|x|<R
\end{equation}
and $\G_R(x,y)=0 $ elsewhere. The constant $C_{N,s}$ is chosen such
that
$$
\int_{\R^N}\G_R(x,y)\,dy = \int_{\R^N \setminus B_R} \G_R(x,y)\,dy =
1 \qquad \text{for every $x \in B_R$.}
$$

\begin{lemma}
\label{sec:gener-prop-distr}
Let $R>0$, $g \in \cL^1_s$, and suppose that $g$ is bounded in a
neighborhood of $B_R$. Then the problem
\begin{equation}
  \label{eq:11} \left \{
  \begin{aligned}
  \Ds u &= 0  &&\qquad \text{in $B_R$,}\\
  u&=g && \qquad \text{in $\R^N \setminus B_R$}
  \end{aligned}
\right.
\end{equation}
has a unique solution $u \in \cL^1_s \cap C^2(B_R) \cap L^\infty(B_R)$ given by
\begin{equation}
  \label{eq:12}
 u(x)=\int_{\R^N\setminus B_R}\G_R(x,y)g(y)dy\qquad
 \text{for $x\in B_R$.}
\end{equation}
Moreover, $u \in C^\infty(B_R)$.
\end{lemma}

\begin{proof}
Let $u: \R^N \to \R$ be defined by $u \equiv g$ in $\R^N \setminus
B_R$ and by (\ref{eq:12}) in $B_R$. Using the explicit representation
(\ref{eq:Poisson-ker-BR}) and the assumption that $g \in \cL^1_s$ is
bounded in a neighborhood of $B_R$, it is easy to see that $u \in
\cL^1_s \cap C^\infty(B_R) \cap L^\infty(B_R)$, and that
$$
\Ds u(x)=
a_{N,s}\lim_{\e\to0}\int_{|x-y|>\e}\frac{u(x)-u(y)}{|x-y|^{N+2s}}dy=0\qquad
\text{for all $x\in B_R.$}
$$
Conversely, let $u\in \cL^1_s \cap C^2(B_R) \cap L^\infty(B_R)$ satisfy
(\ref{eq:11}). Let $r \in (0,R)$ and define $v_r \in C^2(B_r) \cap
L^\infty(\R^N)$ by $v_r \equiv u$ in $\R^N \setminus B_r$ and
\begin{equation}
  \label{eq:13}
v_r(x)= \int_{\R^N\setminus B_r}\G_r(x,y)u(y)dy\qquad
 \text{for $x\in B_r$.}
\end{equation}
Since $u$ is continuous in $B_R$, it is not difficult to see from (\ref{eq:13}) that $v_r \in
C(B_R)$. Applying Lemma~\ref{lem:mmpC2} to $\Omega=B_r$ and the
functions $u-v_r$, $v_r-u$ which are continuous on $\R^N$ and
vanish on $\R^N \setminus B_r$, we infer that $u \equiv v_r$. Passing to
the limit $r \mapsto R^-$ and using the fact that $u$ is bounded in a
neighborhood of $B_R$, we conclude that
$$
u(x)= \lim_{r \mapsto R^-} \int_{\R^N\setminus B_r}\G_r(x,y)u(y)dy=
\int_{\R^N\setminus B_R}\G_R(x,y) u(y) dy = \int_{\R^N\setminus B_R}\G_R(x,y) g(y) dy,
$$
as claimed.
\end{proof}

Next, we wish to remove the $C^2$-assumption in
Lemma~\ref{lem:mmpC2}. For this we consider the regularization of
$\G_R$ as defined in \cite{BB-Schr}. Let $\chi\in C^\infty_c(1/2,1)$ such that $\int_{1/2}^1\chi(r)dr=1$ and define
$$
\tilde{\G}: \R^N \to \R, \qquad \tilde{\G}(y)=\int_{1/2}^1\chi(r)\G_r(0,y)dr.
$$
We have $\tilde{\G}\in C^\infty(\R^N) $ and
\be\label{eq:estC2tG}
| \tilde{\G}(y)| \leq \frac{C_{N,s}}{1+|y|^{N+2+2s}} \qquad
\text{for $y \in \R^N$}
\ee
with a constant $C_{N,s}$, see for instance \cite[Lemma 3.11]{BB-Schr}. For $\eps>0$, we define
$$
\tilde{\G}_\e: \R^N \to \R, \qquad
\tilde{\G}_\e(x)=\e^{-N}\tilde{\G}(x/\e).
$$
The following result is given in \cite[Theorem 3.12]{BB-Schr} for the
case $N \ge 2$, but the same proof also gives the result for $N=1$. For the convenience of the reader, we
include the proof here.

\begin{Theorem}\label{thm:harm-dist}
Let $\Omega \subset \R^N$ be open, and let
$u\in \cL^1_s$ satisfy $\Ds u=0$ in $\O$, i.e.,
\be
\label{eq:vindistrD} \int_{\R^N}u\Ds \psi\, dx=0\qquad \text{for all
$\psi\in C^\infty_c(\O)$.}
\ee Then for every $\eps>0$ we have
\be
\label{eq:uharmsmae}
u \equiv \tilde{\G}_\e*u\quad \textrm{a.e. in $\O_\e$,}
\ee
where $\O_\e:=\{x\in \O\,:\,dist(x,\de \O)>\e\}$. In particular, $u$ is equivalent to a $C^\infty$-function in $\O$.
\end{Theorem}
\begin{proof}
We first remark that, by
Lemma~\ref{sec:gener-prop-distr} and Fubini's theorem, the equality
\eqref{eq:uharmsmae} holds under the additional assumption $u \in
C^2(\O)$. Let $\rho_n \in C^\infty_c(B_{\frac{1}{n}})$ denote the standard
  radially symmetric  mollifier for $n \in \N$. For every $\psi\in C^\infty_c(\O_{\frac{1}{n}})$, we then have
$\rho_n*\Ds \psi =\Ds [\rho_n * \psi]$ in $\R^N$ and therefore, by Fubini's
theorem and \eqref{eq:vindistrD},
$$
\int_{\R^N}[\rho_n*u]\, \Ds \psi \,dx= \int_{\R^N} u\, [\rho_n*\Ds
\psi]\,dy=
  \int_{\R^N}u\, \Ds[ \rho_n*\psi]\,dy=0.
$$
Letting $u_n=\rho_n*u \in C^\infty(\R^N)$ for $n \in \N$, we deduce that $\Ds u_n \equiv
0$ in $\O_{\frac{1}{n}}.$ By the remark above, we thus have $u_n \equiv
\tilde{\G}_\e*u_n$ in $\O_{\e+\frac{1}{n}}$. By \eqref{eq:estC2tG}, we
now may pass to the limit $n \to \infty$ to get
$u \equiv \tilde{\G}_\e*u$ a.e. in $\O_{\e}$, as claimed.
\end{proof}

\begin{Corollary}\label{lem:mmp}
 Let $\Omega \subset \R^N$ be open and
$u\in  \cL^1_s\cap C(\O)$ such that $\Ds u=0 $ in $\O$. Then if $u$
attains it's maximum or it's minimum in $\O$ it is constant.
\end{Corollary}

\begin{proof}
By Theorem \ref{thm:harm-dist} and the continuity of $u$, we have $u
\in C^\infty(\O)$, so the result follows from Lemma~\ref{lem:mmpC2}.
\end{proof}

We finally obtain the following result which improves
Lemma~\ref{sec:gener-prop-distr}.

\begin{lemma}
\label{sec:gener-prop-distr-b}
Let $R,\delta>0$, and let $u \in \cL^1_s \cap
L^\infty(B_{R+\delta})$ satisfy $\Ds u = 0$ in $B_R$. Then
$$ u(x)=\int_{\R^N\setminus B_R}\G_R(x,y)u(y)dy\qquad
 \text{for a.e. $x\in B_R$.}
$$
Moreover, $u$ is equivalent to a $C^\infty$-function in $B_R$.
\end{lemma}

\begin{proof}
By Theorem \ref{thm:harm-dist} we may assume that $u \in
C^\infty(B_R)$. Hence the result follows from Lemma~\ref{sec:gener-prop-distr}.
\end{proof}

Next we consider the Green function associated with $\Ds$ and the unit
ball $\B$, which was computed
by Blumenthal, Getoor and  Ray  in \cite{BGR}. It is given by
\begin{align*}
G_1(x,y) &= k_N^s |x-y|^{2s-N} \int_1^{(\psi(x,y)+1)^{1/2}} \frac{(z^2-1)^{s-1}}{z^{N-1}}\,dz \\
&= \frac{k_N^s}{2}|x-y|^{2s-N}\int_0^{\psi(x,y)}
\frac{z^{s-1}}{(z+1)^{N/2}}\,dz \; \mbox{ with } \;
\psi(x,y)=\frac{(1-|x|^2)(1-|y|^2)}{|x-y|^2}
\end{align*}
for $x,y\in \B$ and $G(x,y)=0$ if $x \not \in \B$ or $y \not \in B$.
Here the normalization constant is given by $k_N^s=\pi^{-(N/2+1)}\Gamma(N/2)\sin(\pi s)$, see \cite{BGR}. If $N=1=2s$, then direct
computations give
$$
\int_0^{\psi(x,y)}
\frac{z^{-1/2}}{(z+1)^{1/2}}\,dz=2\log\frac{ 1-xy+ (1-x^2)^{1/2}(1-y^2)^{1/2}}{|x-y|}
$$
and yet, see \cite{BGR}, in this case
\be\label{eq:G-est-1d}
G_1(x,y)=\frac{1}{\pi} \log\frac{ 1-xy+ (1-x^2)^{1/2}(1-y^2)^{1/2}}{|x-y|}.
\ee

The explicit form of $G_1$ gives rise to the following estimates for
$x,y \in B$:
\be\label{eq:G-est}
G_{\O}(x,y)\leq
\begin{cases}
  C |x-y|^{2s-N}\min\left( \frac{ d^s(x) d^s (y)}{|x-y|^{2s}}, 1\right)  \qquad\textrm{ if } N>2s;\\
C  \min \left(\frac{d^{1/2}(x)d^{1/2}(y)}{|x-y|},\log \frac{3}{|x-y|}\right) \qquad\textrm{ if } N=1=2s;\\
C  |x-y|^{2s-1}\min\left(\frac{(d(x)d(y))^{(2s-1)/2}}{|x-y|^{2s-1}}, \frac{ d^s(x) d^s (y)}{|x-y|^{2s}}  \right) \qquad\textrm{ if } N=1<2s.
\end{cases}
\ee Here $z \mapsto d(z)=1-|z|$ is the distance function to
$\R^N\setminus B$, and $C$ is a constant depending on $N$ and $s$.
Similar estimates are available for Greens functions in general
$C^{1,1}$-domains, see e.g. \cite{K,Chen-Song}. By dilation, the
Green function for the ball $B_R= \{x \in \R^N\::\: |x| <R\}$, $R>0$
is given by
\begin{equation}
  \label{eq:14}
G_R(x,y) = R^{2s-N}G_1\left(\frac{x}{R},\frac{y}{R}\right)
=\frac{k_N^s}{2}|x-y|^{2s-N}\int_0^{\psi_R(x,y)}
\frac{z^{s-1}}{(z+1)^{N/2}}\,dz
\end{equation}
with $\psi_R(x,y)=\frac{(R^2-|x|^2)(R^2-|y|^2)}{R^2|x-y|^2}$. In the next section, we will need the following general Green-Poisson
representation formula.

\begin{Corollary}\label{cor:represeBR+}
Let $R,\delta>0$ and $f\in  L^\infty(B_R)$. Moreover, let $v \in
\cL^1_s \cap L^\infty(B_{R+\delta})$ satisfy $\Ds v=f$ in $B_R$. Then
\begin{equation}
\label{eq:15-1}
v(x)=  \int_{\R^N\setminus B_R}\G_R(x,y)v(y)dy +\int_{B_R} G_R(x,y)
f(y)\,dy \quad \textrm{ for a.e. $x \in B_R$}.
\end{equation}
\end{Corollary}

\begin{proof}
Without loss of generality, we may assume that $R=1$. Consider
$$
w_0: \R^N \to \R,\qquad  w_0(x)=
\int_{\B}G_1(x,y)f(y)dy.
$$
The estimates~\eqref{eq:G-est} imply that
$w_0 \in L^\infty(\R^N)$. Moreover, $\Ds w_0 = f$ in $\B$ in
distributional sense, since for any $\phi \in C^\infty_c(\B)$ we have
$$
\phi(x) = \int_{\B}G_1(x,y)[\Ds \phi](y) dy \qquad \text{for $x \in \B$.}
$$
We now consider $w := v-w_0  \in \cL^1_s \cap
L^\infty(B_{R+\delta})$. Clearly  $\Ds w=0$
 in $B_R$, and $w \equiv v$ on $\R^N\setminus B_R$. By
 Lemma~\ref{sec:gener-prop-distr-b}, we have
$w(x)=\int_{\R^N\setminus B_R}\G_R(x,y)v(y)dy$ for a.e. $x \in
B_R$, and thus (\ref{eq:15-1}) follows.
\end{proof}

We finally add the following boundary estimate.

\begin{lemma}
\label{sec:preliminaries-1}
Let $f \in L^\infty(\B)$ and consider
$$
v: \R^N \to \R,\qquad  v(x)=
\int_{\B}G_1(x,y)f(y)dy.
$$
Then there exists a constant $C=C(N,s)>0$ such that for $x \in \B$ we have
\begin{equation}
\label{eq:7}
v(x) \le C \left\{
\begin{aligned}
&(1-|x|)^s\, \|f\|_{L^\infty} &&\qquad \text{if $2s \le N$;}\\
&(1-|x|)^{s-\frac{1}{2}}\,\|f\|_{L^\infty} &&\qquad \text{if
  $N=1<2s$.}
\end{aligned}
\right.
\end{equation}
\end{lemma}

\begin{proof}
The second inequality in (\ref{eq:7}) is an immediate consequence of
the third inequality in \eqref{eq:G-est}. To prove the first inequality in (\ref{eq:7}), we let $d(x)=1-|x|$ and $B^x:= \{y \in
\R^N\::\: |y-x|<\frac{d(x)}{2}\} \subset \B$ for $x \in \B$. We then have
\begin{equation}
  \label{eq:8}
v(x) =\int_{B^x}G_1(x,y)f(y)dy+\int_{\B \setminus
  B^x}G(x,y)f(y)dy \qquad \text{for $x \in \B$}
\end{equation}
and
\begin{equation}
  \label{eq:15}
d(y)\leq |x-y|+d(x)\leq 3 |x-y| \qquad \text{for $y\in \R^N \setminus B^x.$}
\end{equation}
In the following, the letter $C$ stands for positive constants
depending only on $N$ and $s$. We first consider the
case $N > 2s$. Then the first inequality in \eqref{eq:G-est} implies
that
$$
\Bigl|\int_{B^x}G_1(x,y)f(y)dy\Bigr| \leq C \| f\|_{
  L^\infty}\int_{B^x}|x-y|^{2s-N} dy\leq  C d^{2s}(x) \| f\|_{
  L^\infty},
$$
and together with (\ref{eq:15}) it also yields
\begin{align*}
\Bigl|\int_{\B \setminus B^x}G(x,y)f(y)dy \Bigr|\leq  C d^s(x) \|f\|_{L^\infty}
\int_{\B \setminus B^x}&\frac{d^s(y)}{|x-y|^{N}} dy \le C
  d^s(x)\|f\|_{L^\infty} \int_{\B}|x-y|^{s-N} dy\\
&\le C d^s(x) \|f\|_{L^\infty}
 \int_{\B}|y|^{s-N} dy = C d^s(x)\|f\|_{L^\infty}.
\end{align*}
Combining these two inequalities, we obtain the assertion in the case $N>2s$.\\
Next we consider the case $N=1=2s$. Applying the second estimate in
\eqref{eq:G-est} yields
$$
\Bigl|\int_{B^x}G_1(x,y)f(y)dy\Bigr| \leq  C \|f\|_{L^\infty}
\int_{B^x}
\log \frac{3}{|x-y|}\,dy \le C \|f\|_{L^\infty} \int_0^{\frac{d(x)}{2}}\log
\frac{3}{t}\,dt \le C d^{1/2}(x)\|f\|_{L^\infty}
$$
and, as before, using (\ref{eq:15})
\begin{align*}
\Bigl|\int_{\B \setminus B^x}G_1(x,y)f(y)dy \Bigr|\le
C d^{1/2}(x)\|f\|_{L^\infty} &\int_{\B \setminus
  B^x}\frac{d^{1/2}(y)}{|x-y|}\,dy \le C d^{1/2}(x) \|f\|_{L^\infty} \int_{\B \setminus
  B^x}|x-y|^{-\frac{1}{2}}\,dy\\
&\le C
d^{1/2}(x)\|f\|_{L^\infty} \int_{0}^1 t^{-\frac{1}{2}}\,dt= C
d^{1/2}(x)\|f\|_{L^\infty}.
\end{align*}
Combining these two inequalities, we obtain the assertion in the case $N=1=2s$.
\end{proof}

\section{Green representation on the half-space}
\label{sec:representation}

The purpose of this section is to state
conditions on a function $u$ on the half-space $\R^N_+$ under which
the Green representation formula
$$
u(x)= \int_{\R^N_+} G_\infty^+(x,y) \Ds u(y)\,dy,\qquad  x \in \R^N_+
$$
holds, where $G_\infty^+$ is the half space Green function given by
\begin{equation}
  \label{eq:half-space-green-function}
G_\infty^+(x,y) = \frac{k_N^s}{2}|x-y|^{2s-N}\int_0^{\psi_\infty(x,y)}
\frac{z^{s-1}}{(z+1)^{N/2}}\,dz \quad \mbox{ with } \quad \psi_\infty(x,y)
=\frac{4x_1y_1}{|x-y|^2}
\end{equation}
for $x,y\in \R^N_+$. More precisely, we have the following

\begin{Theorem}
\label{green-poiss-form_plus}
Let  $f\in  L^\infty(\R^N_+)$ be nonnegative, and suppose that $u \in L^\infty(\R^N)$ satisfies
$$
\Ds u= f \quad\textrm{ in }\R^N_+,\qquad \, u=0\quad\textrm{ in }\R^N \setminus\R^N_+.
$$
Then $u$ is continuous, and
\begin{equation}
u(x)= \int_{\R^N_+}G^+_\infty(x,y)f(y)\:dy \qquad \text{for every $x
  \in \R^N_+$.}
\label{rep_plus}
\end{equation}
Moreover, there exist constants $C>0$, $\alpha \in (0,1)$ depending
only on $N$, $s$, $\|f\|_{L^\infty}$, $\|u\|_{L^\infty}$
such that
\begin{equation}
  \label{eq:5}
0 \le u(x) \le C x_1^\alpha \qquad \text{for every $x \in \R^N_+$.}
\end{equation}
\end{Theorem}

The remainder of this section will be devoted to the proof of this
Theorem. We first show how $G_\infty^+$, as defined in
(\ref{eq:half-space-green-function}), arises via an approximation with balls. For this we let $P_R:=(R,0,\dots,0) \in \R^N_+$, and we consider the
translated ball $B_R^+:= \{x \in \R^N\::\: |x-P_R|<R\} \subset \R^N_+$ for $R>0$. By
(\ref{eq:14}), its Green function $G_R^+$ is given by
$$
G_R^+(x,y) = \frac{k_N^s}{2}|x-y|^{2s-N}\!\!\int_0^{\psi^+_R(x,y)}
\!\frac{z^{s-1}}{(z+1)^{N/2}}\,dz \qquad \text{for $x,y\in B_R^+$}
$$
with
\begin{equation}
  \label{eq:6}
\psi_R^+(x,y) =\frac{(R^2-|x-P_R|^2)(R^2-|y-P_R|^2)}{R^2|x-y|^2}=
\frac{(2x_1-\frac{|x|^2}{R})(2y_1-\frac{|y|^2}{R})}{|x-y|^2},
\end{equation}
and its Poisson kernel is given by
\begin{equation}
  \label{eq:4}
\G^+_R(x,y)=C_{N,s}
\left(\frac{R^2-|x-P_R|^2}{|y-P_R|^2-R^2} \right)^s|x-y|^{-N}
\qquad \text{for $|x|<R<|y|$.}
\end{equation}
Note that
\begin{equation}
  \label{eq:1}
B_R^+ \subset B_{R'}^+ \quad \text{if $R'>R>0$} \qquad
\text{and}\qquad
\bigcup_{R>0} B_R^+ = \R^N_+.
\end{equation}

\begin{Lemma}\label{lem:GincmontoG+}
Let $x,y \in \R^N_+$ and $R_0>0$ with $x,y \in B_{R_0}^+$. Then
$G_{R'}^+(x,y) \ge G_{R}^+(x,y)$ for $R' \ge R \ge R_0$ and
$G_R^+(x,y) \to G_\infty^+(x,y)$ as $R \to \infty$.
\end{Lemma}

\begin{proof}
From (\ref{eq:6}) we immediately deduce that $\psi_{R'}^+(x,y) \ge
\psi_{R}^+(x,y)$ for $R' \ge R \ge R_0$, and that
$\psi_R^+(x,y) \to \psi_\infty(x,y)$ as $R \to \infty$. This shows the
claim.
\end{proof}

We may now complete the\\

\begin{altproof}{Theorem~\ref{green-poiss-form_plus}}
It follows from Lemma~\ref{lem:intest} below that $u$ is continuous in
$\R^N_+$. Since $G^+_\infty$ is invariant under translations of the form $(x,y)
\mapsto (x+z,y+z)$ with $z \in \{0\} \times \R^{N-1}$, it suffices to
show (\ref{rep_plus}) for $x=(x_1,0,\dots,0)$ with $x_1>0$. We will
fix such a point $x$ from now on. By Corollary  \ref{cor:represeBR+}
and the fact that $u \equiv 0$ in $\R^N \setminus \R^N_+$, we have
\begin{align}
u(x)=&  \int_{\R^N_+\setminus B_R^+}\G_R^+(x,y)u(y)dy +\int_{B_R^+} G_R^+(x,y) f(y)\,dy \quad \textrm{ for $x \in B_R^+$}.\nonumber
\end{align}
Thanks to Lemma  \ref{lem:GincmontoG+}, the nonnegativity of $f$ and
monotone convergence, (\ref{rep_plus}) follows once
we have shown that
\be
\label{sec:green-repr-half-2}
\int_{\R^N_+\setminus B_R^+}\G^+_R(x,y)u(y)dy \to 0\quad \textrm{  as } R\to\infty.
\ee
In the following, we assume that $R>2x_1$, and we let $C$ denote
(possibly different) constants which may depend on $N,s,u$ and $x_1$ but not on
$R$. Using the fact that $u$ is bounded, we have
\begin{align}
\Bigl|\int_{\R^N_+\setminus B_R^+}&\G^+_R(x,y)u(y)dy\Bigr|\leq C\int_{\R^N_+\setminus B_R^+} \G^+_R(x,y)dy=C
\int_{\R^N_+\setminus B_R^+}\left(\frac{R^2-|x-P_R|^2}{|y-P_R|^2-R^2} \right)^s|x-y|^{-N}dy \nonumber\\
&=  C (2x_1 R-x_1^2)^s \int_{\R^N_+\setminus
  B_R^+}|x-y|^{-N}(|y|^2-2y_1R)^{-s}dy\nonumber\\
&\leq C R^s\int_{\R^N_+\setminus
  B_R^+}|x-y|^{-N}(|y|^2-2y_1R)^{-s}dy=C R^{-s}\int_{\R^N_+\setminus B_1^+}
|y-\frac{x}{R}|^{-N}(|y|^2-2y_1)^{-s}dy. \label{sec:green-repr-half-1}
\end{align}
We will now show that, as $R \to \infty$,
\begin{equation}
\label{eq:17}
\int_{\R^N_+\setminus B_1^+}
|y-\frac{x}{R}|^{-N}(|y|^2-2y_1)^{-s}dy = \left \{
  \begin{aligned}
  &O(1) &&\qquad \text{if $0<s < \frac{1}{2}$;}\\
  &O(R^{2s-1}) &&\qquad \text{if $\frac{1}{2}<s<1$;}\\
  &O(\log R^2) && \qquad \text{if $s=\frac{1}{2}$.}
  \end{aligned}
\right.
\end{equation}
Together with (\ref{sec:green-repr-half-1}) this implies
\eqref{sec:green-repr-half-2}, since $s<1$. To show (\ref{eq:17}),
we decompose the domain of integration as
$$
R^\N_+ \setminus B_1^+ = A_1 \cup A_2\qquad \quad \text{with $\quad A_1:=
  ((0,1) \times \R^{N-1}) \setminus B_1^+,$ $\quad A_2:=
([1,\infty) \times \R^{N-1}) \setminus B_1^+$.}
$$
We then have $|y-\frac{x}{R}|\ge |y-e_1|$ for every $y \in A_2$, where
$e_1=(1,0,\dots,0)$, and therefore
\begin{align}
\int_{A_2}|y-\frac{x}{R}|^{-N}&(|y|^2-2y_1)^{-s}dy  \le
\int_{A_2}|y-e_1|^{-N}(|y-e_1|^2-1)^{-s}dy \nonumber \\
& \le \int_{\R^N \setminus B_1(0)}|y|^{-N}(|y|^2-1)^{-s}dy
 = \omega_{N-1} \int_1^\infty \tau^{-1}(\tau^2-1)^{-s}\,d\tau <
 \infty.\label{sec:green-repr-half}
\end{align}
In case $N=1$, $A_1$ is empty, and thus (\ref{eq:17}) follows. We now
assume that $N \ge 2$, and we put $B_t:=  \{\tilde y \in \R^{N-1} \::\: |\tilde y|^2 \ge 2t -t^2\}$
for $t \in (0,1)$. By Fubini's theorem, we have
\begin{align}
\int_{A_1}
&|y-\frac{x}{R}|^{-N}(|y|^2-2y_1)^{-s}dy=\int_{0}^1 \int_{B_t}(|\tilde
y|^2+(t-\frac{x_1}{R})^2)^{-\frac{N}{2}}(|\tilde y|^2+
t^2-2t)^{-s}d \tilde y d t\nonumber\\
&=\omega_{N-2} \int_{0}^1 \int_{\sqrt{2t -t^2}}^\infty \tau^{N-2}
 (\tau^2 +(t-\frac{x_1}{R})^2)^{-\frac{N}{2}}(\tau^2-[2t-
t^2])^{-s}d\tau d t\nonumber\\
&=\omega_{N-2} \int_{0}^1(2t-t^2)^{-\frac{1}{2}-s}
 \int_{1}^\infty \sigma^{N-2}
 \Bigl(\sigma^2 +\frac{(t-\frac{x_1}{R})^2}{2t-t^2}\Bigr)^{-\frac{N}{2}}(\sigma^2-1)^{-s}d\sigma d t\nonumber\\
&\le \omega_{N-2} \int_{0}^1 t^{-\frac{1}{2}-s}  \int_0^\infty (\rho+1)^{\frac{N-3}{2}}
\Bigl(\rho+1 +\frac{(t-\frac{x_1}{R})^2}{2t}\Bigr)^{-\frac{N}{2}}
\rho^{-s}d\rho
 d t \nonumber\\
&\le \omega_{N-2} \int_{0}^1 t^{-\frac{1}{2}-s}\Bigl(1
+\frac{(t-\frac{x_1}{R})^2}{2t}\Bigr)^{-\frac{1}{2}}  \int_0^\infty
(\rho+1)^{-1}\rho^{-s}d\rho
 d t \nonumber\\
&\le C \int_{0}^1 t^{-\frac{1}{2}-s}\Bigl(1
+\frac{(t-\frac{x_1}{R})^2}{2t}\Bigr)^{-\frac{1}{2}} d t, \label{sec:green-repr-half-4}
\end{align}
whereas
$$
\Bigl(1
+\frac{(t-\frac{x_1}{R})^2}{2t}\Bigr)^{-\frac{1}{2}} \le \sqrt{2t}\Bigl(2t+(t-\frac{x_1}{R})^2
\Bigr)^{-\frac{1}{2}}
\le  \min \Bigl \{1,
\frac{\sqrt{2t}R}{x_1}\Bigr\}
\qquad \text{for $t \in (0,1)$,}
$$
since $\frac{x_1}{R} \le 1$ by
assumption. Inserting this in (\ref{sec:green-repr-half-4}) yields
$$
\int_{A_1}
|y-\frac{x}{R}|^{-N}(|y|^2-2y_1)^{-s}dy  \le
C\Bigl( \frac{\sqrt{2} R}{x_1} \int_0^{\frac{x_1^2}{2 R^2}} t^{-s}\,ds +
 \int_{\frac{x_1^2}{2 R^2}}^1 t^{-\frac{1}{2}-s}\,dt \Bigr)\le
C  \left \{
   \begin{aligned}
&1+R^{2s-1}&&\quad \text{if $s\not=\frac{1}{2}$;}\\
   &1+\log R^2 &&\quad \text{if $s=\frac{1}{2}$.}
   \end{aligned}
\right.
$$
Combining the last inequality with
(\ref{sec:green-repr-half}), we deduce (\ref{eq:17}), as
required. Thus the proof of (\ref{rep_plus}) is finished.\\
To show (\ref{eq:5}), we may assume without loss that
$x=(x_1,0,\dots,0)$ with $0 \le x_1 \le 1$. Moreover, we let $C>0$
denote constants depending on $N,s$, $\|f\|_{L^\infty(\R^N_+)}$ and
$\|u\|_{L^\infty(\R^N_+)}$ but not on $x$. By Corollary
\ref{cor:represeBR+}, we have
$$
u(x) = \int_{\R^N \setminus B_1^+}\Gamma_1^+(x,y)u(y)\,dy
+v(x) \qquad \text{with}\quad  v(x)=
\int_{B_1^+}G_1^+(x,y)f(y)\,dy.
$$
By Lemma~\ref{sec:preliminaries-1}, it suffices to show
(\ref{eq:5}) for $v$ in place of $u$. For this we estimate, similarly
as in (\ref{sec:green-repr-half-1}),
\begin{align}
0 \le v(x) \leq C\int_{\R^N_+\setminus B_1^+} \G^+_1(x,y)dy &\le  C (2x_1 -x_1^2)^s \int_{\R^N_+\setminus
  B_1^+}|x-y|^{-N}(|y|^2-2y_1)^{-s}dy \nonumber\\
&\leq C x_1^s\int_{\R^N_+\setminus
  B_1^+}|y-x|^{-N}(|y|^2-2y_1)^{-s}dy.
\label{eq:21}
\end{align}
Replacing $\frac{x}{R}$ with $x=(x_1,0,\dots,0)$ in (\ref{eq:17}), we
also have, as $x_1 \to 0^+$,
\begin{equation}
\label{eq:22}
\int_{\R^N_+\setminus B_1^+}
|y-x|^{-N}(|y|^2-2y_1)^{-s}dy = \left \{
  \begin{aligned}
  &O(1) &&\qquad \text{if $0<s < \frac{1}{2}$;}\\
  &O(x_1^{1-2s}) &&\qquad \text{if $\frac{1}{2}<s<1$;}\\
  &O(\log \frac{1}{x_1^{2}}) && \qquad \text{if $s=\frac{1}{2}$.}
  \end{aligned}
\right.
\end{equation}
Combining (\ref{eq:21}) and (\ref{eq:22}) yields constants $C>0$,
$\alpha \in (0,1)$ such that (\ref{eq:5}) holds for $v$ in place of
$u$, and this finishes the proof.
\end{altproof}

\section{Proof of the monotonicity result}
\label{sec:proof-monot-result}
In this section we complete the proof of Theorem~\ref{main-result-1}. We know from Theorem \ref{green-poiss-form_plus} that every bounded   solution $u$ of
\eqref{eq:main-half-space} obeys the integral representation
$$
u(x)= \int_{\R^N_+} G^+_\infty(x,y)f(u(y))\,dy \qquad \text{for all $x \in \R^N_+$},
$$
where $ G^+_\infty$ is the half-space Green function given by
(\ref{eq:half-space-green-function}). We note the following simple estimate.

\begin{lemma}
\label{sec:introduction-1}
Let $L>0$. Then there exists a constant $C=C(N,s,L)>0$ such that for
every $x,y \in \R^N_+$ with $x_1,y_1 \le L$ we have
\begin{equation}
  \label{eq:9}
 G^+_\infty(x,y) \le C\left \{
   \begin{aligned}
&\min \Bigl\{|x-y|^{2s-N},|x-y|^{-N}\Bigr\} &&\qquad \text{if $N >
  2s$;}\\
&1 &&\qquad \text{if $N=1< 2s$;}\\
&1+\log \frac{L}{|x-y|} &&\qquad \text{if $N=1= 2s$.}
   \end{aligned}
\right.
\end{equation}
\end{lemma}

\begin{proof}
In the case $N > 2s$ we have $\int_0^{\infty}
\frac{z^{s-1}}{(z+1)^{N/2}}\,dz<\infty$, which immediately implies
that $G^+_\infty(x,y) \le C |x-y|^{2s-N}$ for all $x,y \in \R^N_+$
with a constant $C>0$ depending only on $N$ and $s$.
Moreover, for $t>0$ we have
$$
\int_0^t \frac{z^{s-1}}{(z+1)^{N/2}}\,dz <\frac{t^s}{s}
$$
and
$$
\int_0^t \frac{z^{s-1}}{(z+1)^{N/2}}\,dz \le
\left \{
\begin{aligned}
&\frac{t^{s-N/2}}{s-N/2} &&\qquad \text{if $N<2s$;}\\
&\frac{1}{s} + \log \max \{1,t\} &&\qquad \text{if $N=2s$.}
\end{aligned}
\right.
$$
For $x,y \in
\R^N_+$ with $x_1,y_1 \le L$ we have also $\psi_\infty(x,y) \le
\frac{L^2}{|x-y|^2}$ and therefore
$$
\int_0^{\psi_\infty(x,y)} \frac{z^{s-1}}{(z+1)^{N/2}}\,dz \le
C |x-y|^{-2s}
$$
as well as
$$
\int_0^{\psi_\infty(x,y)} \frac{z^{s-1}}{(z+1)^{N/2}}\,dz \le
C \left \{
  \begin{aligned}
&|x-y|^{N-2s} &&\qquad \text{if $N<2s$,}\\
&1 +\ln \frac{L}{|x-y|}&&\qquad \text{if $N<2s$}
  \end{aligned}
\right.
$$
with a constant $C>0$ depending only on $N,s,L$. This readily implies
the assertion.
\end{proof}

We need some notation. For $\lambda \ge 0$, we consider the set
$$
\Sigma_\lambda= \{x \in \R^N\::\: 0<x_1 <\lambda\} \subset \R^N_+.
$$
From Lemma~\ref{sec:introduction-1} we easily deduce the following.

\begin{corollary}
\label{sec:introduction-2}
As $\lambda \to 0$, $\sup \limits_{x \in \Sigma_\lambda}\int_{\Sigma_\lambda} G^+_\infty(x,y)\,dy \to 0$.
\end{corollary}

We also consider the reflection
$x \mapsto x^\lambda:= (2\lambda-x_1,x_2,\dots,x_N)$ at
the hyperplane $\{x_1=\lambda\}$.

We need the following fact which also follows from Lemma~\ref{sec:introduction-1} in a straightforward way.
\begin{corollary}
\label{sec:introduction-2.1}
If $(\lambda_n)_n \subset (0,\infty)$ and $(z^n)_n \subset \R^N_+$ are
sequences with
$$
\lambda_n \to \lambda>0 \quad \text{and}\quad z^n \to z \in \R^N_+ \qquad
\text{as $n \to \infty$,}
$$
then
$$
\int_{\Sigma_{\lambda_n}} G^+_\infty(z^n,y^{\lambda_n})\:dy \to
\int_{\Sigma_{\lambda}} G^+_\infty(z,y^{\lambda})\:dy
\qquad \text{as $n \to \infty$.}
$$
\end{corollary}

In the sequel we also consider
$J_{\lambda}:= \{x \in \R^N_+ \::\: x_1 \ge 2 \lambda\}$.
We need the following ``reflection inequalities'' for the Green function.
\begin{lemma}
\label{green-monotonicity}
We have
 \begin{equation}
\label{eq:Glambda1} \left.
\begin{aligned}
& G^+_\infty(x^\lambda,y^\lambda)>  G^+_\infty(x,y^\lambda)\qquad
\text{and}\\
& G^+_\infty(x^\lambda,y^\lambda)- G^+_\infty(x,y)>  G^+_\infty(x,y^\lambda)- G^+_\infty(x^\lambda,y)
\end{aligned}
\quad \right \} \qquad \text{for all $x,y \in \Sigma_\lambda$}
\end{equation}
and
\begin{equation}
  \label{eq:extra-est1}
 G^+_\infty(x^{{\lambda}},y)- G^+_\infty(x,y)>0 \qquad \text{for $x \in \Sigma_{\lambda}$, $y \in J_\lambda$.}
\end{equation}
\end{lemma}

\begin{proof}
We note that $G(x,y)= H(s(x,y),t(x,y))$, where
$$
H: (0,\infty) \times [0,\infty) \to \R,\qquad H(r,t)=
r^{s-\frac{N}{2}} \int_0^{\frac{t}{r}} \frac{z^{s-1}}{(z+1)^{N/2}}\,dz.
$$
and
$$
r(x,y)=|x-y|^2, \quad t(x,y)=4x_1 y_1 \qquad \text{for $x,y \in \R^N_+$.}
$$
By \cite[Lemma 2]{BGW} we have
\begin{equation}
\label{eq:monotonicity}
\partial_r H(r,t)<0,\quad \partial_t H(r,t)>0\quad \text{and}\quad \partial_r
\partial_t H(r,t)<0 \qquad \text{for $r,t>0$.}
\end{equation}
Hence
\begin{equation}\label{eq:19}
H(r_1,t_1) > H(r_2,t_2)\qquad\text{if $r_1 < r_2,\:t_1 > t_2$,}
\end{equation}
and
\begin{align}
H(r_1,t_4) -H(r_1,t_1)&= \int_{t_1}^{t_4} \partial_t H(r_1,t)\:dt
 > \int_{t_1}^{t_4} \partial_t H(r_2,t)\:dt \nonumber > \int_{\min\{t_2,t_3\}}^{\max\{t_2,t_3\}}
 \partial_t H(r_2,t)\:dt \nonumber\\
&= |H(r_2,t_2)-H(r_2,t_3)| \qquad \qquad \text{if $0<r_1<r_2,\; 0< t_1 < t_2,t_3 < t_4$.} \label{eq:20}
\end{align}
Now fix $\lambda >0$. Then for $x,y \in \Sigma_\lambda$ we have
\begin{equation}\label{eq:18}
r(x,y)= r(x^\lambda,y^\lambda)<r(x,y^\lambda)= r(x^\lambda,y)
\end{equation}
and
\begin{equation}\label{eq:16}
t(x^\lambda,y^\lambda) \,> \,\max\{t(x^\lambda,y),\,t(x,y^\lambda)\}
\,\ge\, \min\{t(x^\lambda,y),\,t(x,y^\lambda)\}\,>\,t(x,y).
\end{equation}
The inequalities given in Lemma~\ref{green-monotonicity} now follow from \eqref{eq:19},~\eqref{eq:20},~\eqref{eq:18} and~\eqref{eq:16}.
\end{proof}

We now fix a solution $u$ of \eqref{eq:main-half-space}, and we let
$C_u>0$ be a Lipschitz constant for $f$ on
$\bigl[0,\|u\|_{L^\infty(\R^N)} \bigr]$, so that
$$
|f(t)-f(r)| \le C_u |t-r| \qquad \text{for all $r,t \in
  \bigl[ 0,\|u\|_{L^\infty(\R^N)}\bigr]$.}
$$
Inequality \eqref{eq:extra-est1} and the nonnegativity of $f$ imply that
\begin{equation}
  \label{eq:new11}
 \int_{J_{\lambda}}[ G^+_\infty(x^{{\lambda}},y)- G^+_\infty(x,y)]\,
f(u(y))\, dy \ge 0  \qquad \text{for $\lambda>0$ and $x \in
\Sigma_{\lambda}$.}
\end{equation}
We claim that the following reflection inequality holds for every $\lambda>0$:
$$
\leqno{(\cC_\lambda)}\qquad \qquad \qquad
u(x)\le u(x^\lambda)\qquad \text{for all $x \in \Sigma_\lambda$.}
$$
As a first step, we prove

\begin{lemma}
\label{sec:introduction-3}
There exists $\lambda_0>0$ such that $(\cC_{\lambda})$ holds for
$\lambda \in [0,\lambda_0]$.
\end{lemma}

\begin{proof}
By Corollary~\ref{sec:introduction-2}, we may fix $\lambda_0>0$ such
that
\begin{equation}
  \label{eq:lambda_0}
\int_{\Sigma_{2\lambda_0}} G^+_\infty(x,y)\,dy<C_u^{-1} \qquad \text{for every $x \in \Sigma_{\lambda_0}$},
\end{equation}
For fixed $\lambda \in [0,\lambda_0]$, we consider the
difference function
$$
v: \Sigma_\lambda \to \R,\qquad v(x)= u(x^\lambda)-u(x)
$$
and the set
$$
W:= \{x \in \Sigma_\lambda \::\:
v(x)<0\}.
$$
For $x \in W$ we estimate, using Lemma~\ref{green-monotonicity} and \eqref{eq:new11},
\begin{align*}
0&>v(x)=\int_{\R^N_+}[ G^+_\infty(x^{\lambda},y)- G^+_\infty(x,y)]f(u(y))\,dy=
\int_{\Sigma_{\lambda}} \ldots \,dy
+ \int_{\R^N_+ \setminus \Sigma_{\lambda}} \ldots \,dy \\
&=
\int_{\Sigma_{\lambda}}\Bigl([ G^+_\infty(x^{\lambda},y)- G^+_\infty(x,y)]f(u(y))+
[ G^+_\infty(x^{\lambda},y^{\lambda})
- G^+_\infty(x,y^{\lambda})]\,f(u(y^{\lambda}))\Bigr)\,dy\\
& +\int_{J_{\lambda}}[G(x^{\lambda},y)- G^+_\infty(x,y)]f(u(y))\, dy \ge \int_{\Sigma_{\lambda}}[ G^+_\infty(x^{{\lambda}},y^{\lambda})
- G^+_\infty(x,y^{\lambda})][f(u(y^{\lambda}))-f(u(y))]\,dy\\
&\ge  \int_{W}[ G^+_\infty(x^{\lambda},y^{\lambda})
- G^+_\infty(x,y^{\lambda})][f(u(y^{\lambda}))-f(u(y))]\,dy
\ge
\int_{W} G^+_\infty(x^{{\lambda}},y^{\lambda})[f(u(y^{\lambda}))-f(u(y))]\,dy\\
&\ge C_u \int_{W} G^+_\infty(x^{{\lambda}},y^{\lambda})v(y)\,dy \ge - C_u
\|v\|_{L^\infty(W)} \int_{W} G^+_\infty(x^{{\lambda}},y^{\lambda})\,dy \ge -C_u \|v\|_{L^\infty(W)}
\int_{\Sigma_{2\lambda_0}} G^+_\infty(x^{{\lambda}},y)\,dy
\end{align*}
Combining this with \eqref{eq:lambda_0} we infer
$\|v\|_{L^\infty(W)} \le C \|v\|_{L^\infty(W)}$ with some constant $C
  \in (0,1)$, hence $v \equiv 0$ on $W$ and therefore $W= \varnothing$
  by definition. We therefore conclude that $(\cC_{\lambda})$ holds.
\end{proof}

Next we put
$$
\lambda_*:= \sup \{\lambda>0\::\: \text{$(\cC_{\lambda'})$ holds for all
    $\lambda' \le \lambda$}\}.
$$
Then $\lambda_* \ge \lambda_0$. Using the continuity of $u$, it is easy to see that
$(\cC_{\lambda_*})$ holds. We suppose by contradiction that
\begin{equation}
\label{eq:lambda_infty}
\lambda_*<\infty
\end{equation}
 Then there exists a sequence of numbers
$\lambda_n>\lambda^*$, $n \in \N$ and points
$x^n \in \Sigma_{\lambda_n}$ such that
\begin{equation}
  \label{eq:contradiction}
u(x^n)>u((x^n)^{\lambda_n})\qquad \qquad \text{for all $n$}
\end{equation}
and
\begin{equation}
  \label{eq:limit}
\lambda_n \to \lambda^* \qquad \text{as $n \to \infty$.}
\end{equation}
We may further assume that
\begin{equation}
  \label{eq:Linfty}
u(x^n)-u((x^n)^{\lambda_n}) > \frac{1}{2} \sup_{x \in
  \Sigma_{\lambda_n}}\Bigl(u(x)-u(x^{\lambda_n})\Bigr).
\end{equation}
In the following we write $x= (x_1,\hat x)$ for $x \in \R^N$ with
$\hat x \in \R^{N-1}$. For $n \in \N$ we define the translated functions
$$
u_n: \R^N \to \R,\qquad u_n(y)= u(y_1,\hat x^n+\hat y)
$$
and
$$
v_n: \Sigma_{\lambda_n} \to \R, \qquad v_n(y)= u_n(y^{\lambda_n})-u(y).
$$
Then \eqref{eq:contradiction} is rewritten as
\begin{equation}
  \label{eq:translated}
v_n(x^n_1,0)= u_n(2\lambda_n - x^n_1,0)-u_n(x^n_1,0)<0 \qquad
\text{for all $n$.}
\end{equation}
We also consider the sets
$$
W_n:=\{x \in \Sigma_{\lambda_n} \::\:
v_n(x)<0\}.
$$
We let $\Lambda:= 2 \max_n \lambda_n<\infty$. For abbreviation, we
also put $z^n=(x^n_1,0) \in \Sigma_{\lambda_n}$
for $n \in \N$. Using~\eqref{eq:Linfty} and arguing similarly as in the proof of Lemma~\ref{sec:introduction-3}, we find
\begin{equation}
  \label{eq:new.extra.1}
-\|v_n\|_{L^\infty(W_n)}\ge 2 v_n(z^n)\ge 2
\int_{W_n}\bigl( G^+_\infty((z^n)^{\lambda_n},y^{\lambda_n})-
 G^+_\infty(z^n,y^{\lambda_n})\bigr)[f(u_n(y^{\lambda_n}))-f(u_n(y))]\,dy
\end{equation}
We now pass to a subsequence such that $x^n_1 \to t \in
[0,\lambda_*]$. Moreover, we note that the sequence $(u_n)_n$ is
uniformly equicontinuous on compact subsets of $\R^N$. Indeed, this
follows from Lemma~\ref{lem:intest} below and the boundary estimate
(\ref{eq:5}) which holds uniformly for $u_n$, $n \in \N$ in place of $u$.
Hence we may pass to a subsequence such that $u_n \to
\bar u$ in $C_{loc}^0(\R^N)\cap L^\infty(\R^N)$ and $\bar u \equiv 0$
on $\R^N \setminus \R^N_+$.
We distinguish three cases.\\
{\bf Case 1:} $t=\lambda_*$, i.e. $z^n \to (\lambda_*,0)$.
In this case \eqref{eq:new.extra.1} implies that
\begin{align}
-\|v_n\|_{L^\infty(W_n)}&\ge 2C_u \int_{W_n}\bigl( G^+_\infty((z^n)^{\lambda_n},y^{\lambda_n})-
 G^+_\infty(z^n,y^{\lambda_n})\bigr)v_n(y)\,dy \nonumber \\
&\ge -  2C_u \|v_n\|_{L^\infty(W_n)}
\int_{\Sigma_{\lambda_n}}\bigl( G^+_\infty((z^n)^{\lambda_n},y^{\lambda_n})-
 G^+_\infty(z^n,y^{\lambda_n})\bigr)\,dy \nonumber \\
&= -  o(1) \|v_n\|_{L^\infty(W_n)}, \label{combine0}
\end{align}
where in the last step we used Corollary~\ref{sec:introduction-2.1}
and the fact that also $(z^n)^{\lambda_n} \to (\lambda_*,0)$.
Hence $\|v_n\|_{L^\infty(W_n)}= 0$ for $n$ large, contrary to
\eqref{eq:translated}.\\
{\bf Case 2:} $t< \lambda^*$ and $\bar u \not \equiv 0$.\\
We have
$$
 \int_{\R^N}u_n(x)\Ds\phi(x)\,dx=\int_{\R^N_+}f(u_n(x))\phi(x)\,dx\qquad
 \text{for all $\phi\in C^2_c(\R^N_+).$}
$$
It thus follows from the dominated convergence theorem that
$$
 \int_{\R^N}\bar{u}(x)\Ds\phi(x)\,dx=\int_{\R^N_+}f(\bar{u}(x))\phi(x)\,dx\qquad\text{for
   all $\phi\in C^2_c(\R^N_+).$}
$$
Hence, by  Theorem
\ref{green-poiss-form_plus}, $\bar{u}$ is represented as
$$
\bar u(x)= \int_{\R^N_+} G^+_\infty(x,y)f(\bar u(y))\,dy \qquad \text{for all $x \in \R^N_+$}.
$$
Since $u \not \equiv 0$ and $f(t)>0$ for $t>0$, it then follows that $\bar u$ and $f
\circ \bar u$ are strictly positive on $\R^N_+$. Moreover, by the locally uniform
convergence $u_n \to \bar u$, we have
$$
\bar u(x)\le \bar u(x^{\lambda_*})\qquad \text{for all $x \in
\Sigma_{\lambda_*}$.}
$$
Indeed this inequality is strict, since we have a strict inequality in
(\ref{eq:new11}) for $\bar u$ in place of $u$ and $\lambda=\lambda^*$,
so that
\begin{align}
\bar u&(x^{\lambda_*})-\bar u(x)=
\int_{\R^N_+}[ G^+_\infty(x^{{\lambda_*}},y)- G^+_\infty(x,y)]f(\bar u(y))\,dy=
\int_{\Sigma_{\lambda_*}} \ldots \,dy
+ \int_{\R^N_+ \setminus \Sigma_{\lambda_*}} \ldots \,dy \nonumber\\
&=
\int_{\Sigma_{\lambda_*}}\Bigl([ G^+_\infty(x^{{\lambda_*}},y)- G^+_\infty(x,y)]f(\bar u(y))+
[ G^+_\infty(x^{{\lambda_*}},y^{\lambda_*})
- G^+_\infty(x,y^{\lambda_*})]\,f(\bar u(y^{\lambda_*}))\Bigr)\,dy \nonumber\\
& +\int_{J_{\lambda_*}}[ G^+_\infty(x^{{\lambda_*}},y)- G^+_\infty(x,y)]f(\bar u(y))\, dy \nonumber\\
& > \int_{\Sigma_{\lambda^*}}[ G^+_\infty(x^{{\lambda_*}},y^{\lambda_*})
- G^+_\infty(x,y^{\lambda_*})]f(\bar u(y^{\lambda_*}))-f(\bar u(y))]\,dy
\ge 0 \quad \text{for $x \in \Sigma_{\lambda_*}$.} \label{eq:--int_s_l}
\end{align}
On the other hand, also by the locally uniform
convergence and \eqref{eq:translated}, we have
$$
\bar u(2\lambda_*-t,0)- \bar u(t,0)= \lim_{n \to
  \infty}\bigl(u_n((z^n)^{\lambda_n})-u(z_n)\bigr) \le 0.
$$
This is a contradiction.\\
{\bf Case 3:} $\bar u \equiv 0$. We use \eqref{eq:new.extra.1} and Lemma~\ref{green-monotonicity} to estimate, for $r>0$,
\begin{align}
- \|v_n\|_{L^\infty(W_n)}
\ge  &2\int_{W_n} G^+_\infty((z^n)^{\lambda_n},y^{\lambda_n})
[f(u_n(y^{\lambda_n}))-f(u_n(y))]\,dy
\nonumber \\
=  &2\int_{W_n \cap B_r(0)} G^+_\infty((z^n)^{\lambda_n},y^{\lambda_n})
[f(u_n(y^{\lambda_n}))-f(u_n(y))]\,dy
\nonumber \\
&+
2\int_{W_n \setminus B_r(0)} G^+_\infty((z^n)^{\lambda_n},y^{\lambda_n})
[f(u_n(y^{\lambda_n}))-f(u_n(y))]\,dy, \label{combine1}
\end{align}
where
\begin{align}
\int_{W_n \setminus B_r(0)}& G^+_\infty((z^n)^{\lambda_n},y^{\lambda_n})
[f(u(y^{\lambda_n}))-f(u(y))]\,dy \ge
C_u \int_{W_n \setminus B_r(0)} G^+_\infty((z^n)^{\lambda_n},y^{\lambda_n})
v_n(u(y))\,dy \nonumber\\
& \ge -C_u \|v_n\|_{L^\infty(W_n)}
\int_{\Sigma_{\Lambda} \setminus B_r(0)} G^+_\infty((z^n)^{\lambda_n},y)\:dy \label{combine2}
\end{align}
Since $((z^n)^{\lambda_n})_n \subset \Sigma_{\Lambda}$ is a bounded sequence,
Lemma~\ref{sec:introduction-1} implies that
$$
\sup_{n \in \N}\int_{\Sigma_{\Lambda} \setminus
  B_r(0)} G^+_\infty((z^n)^{\lambda_n},y)\:dy \to 0 \qquad \text{as $r \to \infty$.}
$$
Hence we may fix $r>0$ such that
\begin{equation}
 \label{combine3}
\int_{\Sigma_{\Lambda} \setminus B_r(0)}G^+_\infty((z^n)^{\lambda_n},y)\:dy
<\frac{1}{4C_u} \qquad \text{for all $n \in \N$}.
\end{equation}
Moreover, assumption \eqref{eq:extra-assumption} and the locally
uniform convergence $u_n \to 0$ imply that there exists
a sequence of numbers $\eps_n>0$, $\eps_n \to 0$ such that
$$
f(u_n(y^{\lambda_n}))-f(u_n(y)) \ge \eps_n  v_n(y) \qquad \text{for $y \in W_n \cap B_r(0)$,}
$$
so that
\begin{align}
\int_{W_n \cap B_r(0)} G^+_\infty((z^n)^{\lambda_n},y^{\lambda_n})
[f(u_n(y^{\lambda_n}))-f(u_n(y))]\,dy &\ge -\eps_n \|v_n\|_{L^\infty(W_n)}
\int_{B_r(0)} G^+_\infty((z^n)^{\lambda_n},y^{\lambda_n})\,dy \nonumber\\
& \ge - C \eps_n \|v_n\|_{L^\infty(W_n)}  \label{combine4}
\end{align}
with some constant $C=C(r)>0$. Combining \eqref{combine1}, \eqref{combine2}, \eqref{combine3} and \eqref{combine4}, we get
$$
\|v_n\|_{L^\infty(W_n)} \le \Bigl(\frac{1}{2}+2C\eps_n \Bigr) \|v_n\|_{L^\infty(W_n)},
$$
so we conclude that $\|v_n\|_{L^\infty}(W_n)= 0$ for large $n$,
contradicting again~(\ref{eq:translated}). We have thus proved that
property $(\cC_\lambda)$ holds for all $\lambda>0$, which implies that
$u$ is increasing in $x_1$. It thus remains to show that $u$ is strictly
increasing in $x_1$ if $u \not \equiv 0$. This however follows since
we can now derive inequality (\ref{eq:--int_s_l}) for $u$ in place
of $\bar u$ and all $\lambda>0$ in place of $\lambda^*$. The proof of
Theorem~\ref{main-result-1} is thus finished.


\section{Proof of the Liouville Theorem in the half-space}
\label{sec:proof-theorem}
This section is devoted to the proof of Theorem~\ref{main-result-2}.
Let $s\in(0,1)$ and assume that $q>1$ if $N\le 1+2s$ and $1<q <
\frac{N-1+2s}{N-1-2s}$ if $N> 1+2s$. Suppose by contradiction that
there exists a bounded solution $u \not \equiv 0$ of
(\ref{eq:main-liouville}). Theorem~\ref{main-result-1} implies that
$u$ is strictly increasing in
$x_1$. In particular, we may define
$$
\tilde u \in L^\infty(\R^{N-1}),\qquad \tilde u(x')= \lim_{x_1 \to
  \infty} \tilde u(x_1,x') \qquad \text{for $x' \in \R^{N-1}$.}
$$
Here and in the following, we write $x=(x_1,x') \in \R^N$ with $x'
\in \R^{N-1}$. Note that $\tilde u>0$ on $\R^{N-1}$. For $\tau>0$,
define $u_\tau \in L^\infty(\R^N) \cap C(\R^N)$ by $u_\tau(x)=
u(x+\tau e_1) $. Then
\begin{equation}
  \label{eq:2}
u_\tau \to u_\infty \qquad \text{pointwise on $\R^N$ with $u_\infty \in L^\infty(\R^N)$ defined by $u_\infty(x)= \tilde
u(x')$.}
\end{equation}
Moreover, we have $\Ds u_\tau = u_\tau^q$ in $H_\tau:=\{x \in \R^N\::\:
x_1 >-\tau\}$. Hence (\ref{eq:2}) implies that
$\Ds u_\infty = u_\infty^q$ in $\R^N$ in distributional sense.
Indeed, let $\vp \in C_c^2(\R^N)$. Then
\begin{align*}
\int_{\R^N}u_\infty^q(x) \vp(x)\,dx = \lim_{\tau \to \infty}
\int_{H_\tau}u_\tau^q(x) \vp(x)\,dx&=\lim_{\tau \to \infty}
\int_{\R^N}u_\tau(x) [\Ds \vp](x)\,dx\\
&= \int_{\R^N}u_\infty(x) [\Ds \vp](x)\,dx
\end{align*}
by Lebesgue's theorem and the estimate \eqref{sec:introduction-4}.
It then follows from the regularity
results in \cite{Sil} and \cite{cabre-sire} that also $\Ds u_\infty=
u_\infty^q$ in classical sense in $\R^N$. Moreover, $u_\infty>0$ in
$\R^N$ since $\tilde u>0$ in $\R^{N-1}$. In case $N=1$, $u_\infty$
is a positive constant on $\R$ which contradicts the equation $\Ds
u_\infty= u_\infty^q$. In case $N \ge 2$, we deduce that
\begin{equation}
  \label{eq:main-liouville-rn-1}
\displaystyle \Ds \tilde u=  {\tilde u}^q \qquad \text{in
$\R^{N-1}$}.
\end{equation}
Indeed,
\begin{align*}
{\tilde u}^q(x')&= {u_\infty}^q(x)= \Ds u_\infty(x) = a_{N,s}
\int_{\R^N}\frac{u_\infty(x+z)+u_\infty(x-z)-2u_\infty(x)}{|z|^{N+2s}}\,dz\\
&= a_{N,s}
\int_{\R^N}\frac{\tilde u(x'+z')+\tilde u(x'-z')-2\tilde u(x')}{(|z'|^2+
  z_1^2)^{\frac{N}{2}+s}}\,dz\\
&= a_{N,s}
\int_{\R^{N-1}}\frac{\tilde u(x'+z')+\tilde u(x'-z')-2\tilde
  u(x')}{|z'|^{N-1+2s}}\:\Bigl[\frac{1}{|z'|} \int_{\R}{\bigl[1+
  \bigl(\frac{z_1}{|z'|}\bigr)^2\bigr]^{-\frac{N}{2}-s}}\,dz_1\Bigr]\,dz'\\
&= \frac{a_{N,s}}{a_{N-1,s}} [\Ds \tilde u](x')
\int_{\R}(1+\lambda^2)^{-\frac{N}{2}-s}\,d\lambda = [\Ds \tilde
u](x')\qquad \text{for $x' \in \R^{N-1}$},
\end{align*}
since
$$
\int_{\R}(1+\lambda^2)^{-\frac{N}{2}-s}\,d\lambda=
B(\frac{1}{2},\frac{N-1}{2}+s)= \frac{\sqrt{\pi} \Gamma(\frac{N-1}{2}+s)}{\Gamma(\frac{N}{2}+s)}=
\frac{a_{N-1,s}}{a_{N,s}}.
$$
However, since $\tilde u >0$, this is impossible by
Theorem~\ref{main-result-3} applied to $N-1 \ge 1$ in place of
$N$. The proof is finished.

\section{Appendix: Interior H\"{o}lder estimates for distributional
  solutions to the equation $\Ds u=f$}

In the proof of Theorem~\ref{main-result-1} in
Section~\ref{sec:proof-monot-result} we used the following lemma on
local H\"older regularity.
Since the proofs available in the literature are only concerned with
the case $N>2s$ and additional assumptions (see for instance
\cite[Proposition 2.1.9]{Sil}), we will give a proof for the
convenience of the reader. The proof is similar to arguments in
\cite[page 263]{lieb-loss}.

\begin{Lemma}\label{lem:intest}
 Let   $f \in
 L^\infty(B_1)$, and let $u\in \cL^1_s\cap L^\infty_{loc}(B_1)$ such that  $\Ds u=f$
in $B_1$. Then for $r\in(0,1)$ and
 every $\a\in(0,\min(1,2s))$ there exists a constant
$C_{s,N,r,\a}>0$ such that \be\label{eq:estuN-int}
\|u\|_{C^{0,\a}(B_r)}\leq C_{s,N,r,\a} \left ( \|u\|_{L^\infty(B_r)}
+\|f\|_{L^\infty(B_1)} \right). \ee
\end{Lemma}

\begin{proof}
Let  $\eta\in C^\infty_c(\R^N)$ be such that $\eta=1$ on
$B_r$, $\eta=0$ on $\R^N\setminus B_1$  and $0 \le \eta\leq 1$ on $\R^N$.
Consider the Riesz potential, see \cite{BGR} and \cite{BKN},
$$
\Phi(x,y)=
\begin{cases}
\displaystyle C_{N,s}|x-y|^{2s-N}\quad \textrm{ for } N>2s,\\
\displaystyle\frac{1}{\pi}\log\frac{1}{|x|} \quad \textrm{ for } 1=N=2s,\\
\displaystyle -C_{1,s}|x-y|^{2s-N}\quad \textrm{ for } N=1<2s,\\
\end{cases}
$$
with a positive normalization constant $C_{N,s}$. Then the function $v(x)=\int_{\R^N}\Phi(x,y)(\eta f)(y) dy$ satisfies
\be\label{eq:v-solv}
\Ds v(x)=\eta(x) f(x)\quad \textrm{ for all } x\in\R^N.
\ee
The following inequalities hold for all $a,b>0$, $\a\in(0,1)$ and
$m\in\R$ with $m+\a>0$:
\be |b^{-m}-a^{-m}|\leq
\frac{|m|}{m+\a}|b-a|^\a\max(  a^{-(m+\a)} ,   b^{-(m+\a)} ) \ee \be
|\log b-\log a |\leq \frac{1}{\a}|b-a|^\a\max(  a^{-\a} ,   b^{-\a}
). \ee
These estimates were proved in \cite[page 263]{lieb-loss} in the case
$m\geq 1$, and a slight change of their argument yields the more
general case considered here. We thus get,
for every $x,y,z\in\R^N$,
$$
\bigl||x-z|^{-m}- |y-z|^{-m} \bigr|\leq \frac{|m|}{m+\a} |x-y|^{\a}\bigl(
|x-z|^{-(m+\a)}+ |y-z|^{-(m+\a)}  \bigr)
$$
and
$$
\bigl|\log|x-z|- \log|y-z| \bigr|\leq \frac{1}{\a} |x-y|^{\a}\bigl( |x-z|^{-\a}+
|y-z|^{-\a}  \bigr).
$$
In the case $N\geq 2s$ we therefore have, for $\a\in(0, \min(2s,1))$ and $x,y\in
B_r$,
\begin{align}\label{eq:estvNg2s}
\nonumber | v(x)-v(y)| &\leq C_{s,\a}|x-y|^\a \int_{\R^N}
|x-y|^{2s-N-\a}\eta(y) |f(y)|dy \\
\nonumber&\leq  C_{s,\a}|x-y|^\a \|f\|_{L^\infty(B_1)} \int_{B(x,2)}|x-y|^{2s-N-\a}dy\\
 &\leq C_{s,\a} \|f\|_{L^\infty(B_1)} |x-y|^\a.
\end{align}
Moreover, in the case $N=1<2s$ we have, for $\a\in(0, \min(2s,1))$ and
$x,y\in B_r$,
\be\label{eq:estvNl2s} | v(x)-v(y)|\leq C_{s,\a}|x-y|^\a
\|f\|_{L^\infty(B_1)} \int_{B_1}
|x-y|^{2s-1-\a}dy\leq
C_{s,\a}\|f\|_{L^\infty(\O)} |x-y|^\a. \ee
 Hence combing
\eqref{eq:estvNg2s} and \eqref{eq:estvNl2s}, we conclude that
\be\label{eq:estvNok} \|v\|_{C^{0,\a}(B_r)}\leq C_{s,N}
\|f\|_{L^\infty(B_1)}, \ee
  for every $\a\in(0,\min(1,2s))$.\\ 
Next we note that the function $w:=u-v$ satisfies $\Ds w = 0$ in $B_r$ by
\eqref{eq:v-solv}. Therefore, thanks to \cite[Lemma 3.2]{BKN}, we
get, for every $r'\in(0,r)$,
$$
\|\n w\|_{L^\infty(B_r')} \leq C_{N,s,r'} \|w\|_{L^\infty(B_1)} \leq
C_{N,s,r'}(\|u\|_{L^\infty(B_r)}+\|v\|_{L^\infty(B_r)}).
$$
From this, together with \eqref{eq:estvNok}, we conclude that
$$
\|u\|_{C^{0,\a}(B_{r' })}=\|w+v\|_{C^{0,\a}(B_{r'})}\leq
C_{s,N,r'}\left( \|u\|_{L^\infty(B_{r})} + \|f\|_{L^\infty(B_1)}
\right),
$$
  for every $\a\in(0,\min(1,2s))$.

\end{proof}


\end{document}